\documentclass[a4paper,11pt,titlepage,twoside]{article}

\usepackage{graphicx}
\usepackage[T1]{fontenc} 
\usepackage[utf8]{inputenc}
\usepackage[english]{babel}
\usepackage{soul}
\usepackage{amsfonts}
\usepackage{amsmath}
\usepackage{amsthm}
\usepackage{amssymb}
\usepackage{mathrsfs}
\usepackage[top=5.5cm, bottom=3cm, left=4cm, right=4cm]{geometry}
\usepackage{setspace}
\usepackage{afterpage}
\usepackage{extarrows}
\usepackage{fancyhdr}
\usepackage{titlesec}
\usepackage{enumitem} \setlist{nosep}
\usepackage[pdftex,breaklinks,colorlinks,linkcolor=blue,
anchorcolor=blue]{hyperref}

\theoremstyle{definition}
\newtheorem{defin}{Definition}[section]
\theoremstyle{plain}
\newtheorem{theo}[defin]{Theorem}
\newtheorem{lem}[defin]{Lemma}
\newtheorem{pro}[defin]{Proposition}
\newtheorem{cor}[defin]{Corollary}
\theoremstyle{definition}
\newtheorem{exm}[defin]{Example}
\newtheorem{rem}[defin]{Remark}

\renewcommand{\O}{\Omega}
\newcommand{\D}{\mathfrak{D}}
\renewcommand{\H}{\mathcal{H}}

\newcommand{\B}{\mathcal{B}}
\newcommand{\dom}{\mathcal{D}}

\newcommand{\Up}{\Upsilon}
\renewcommand{\L}{\Lambda}

\newcommand{\n}[1]{\|#1\|}
\newcommand{\nor}{\|\cdot\|}

\renewcommand{\l}{\langle}
\renewcommand{\r}{\rangle}
\newcommand{\N}{\mathbb{N}}
\newcommand{\R}{\mathbb{R}}
\newcommand{\C}{\mathbb{C}}

\newcommand{\pint}{\l\cdot,\cdot\r}
\newcommand{\pin}[2]{\l#1 , #2\r}
\newcommand{\no}{\noindent}
\newcommand{\ol}{\overline}

\newcommand{\mez}{\frac{1}{2}}
\numberwithin{equation}{section}

\titleformat{\section}
{\normalfont\fillast \fontsize{12}{15}\scshape}{\thesection.}{0.8em}{}

\titleformat{\subsection}
{\normalfont\fillast \fontsize{11}{12}\scshape}{\thesubsection.}{0.8em}{}

\begin{document}
	
\thispagestyle{plain}

\begin{center}
	\large
	{\uppercase{\bf Sesquilinear forms associated \\ to sequences on Hilbert spaces}} \\
	\vspace*{0.5cm}
	{\scshape{Rosario Corso}}
\end{center}

\normalsize 
\vspace*{1cm}	

\small 

\begin{minipage}{11.8cm}
	{\scshape Abstract.} 
	The possibility of defining sesquilinear forms starting from one or two sequences of elements of a Hilbert space is investigated. One can associate operators to these forms and in particular look for conditions to apply representation theorems of sesquilinear forms, such as Kato's theorems. \\
	The associated operators correspond to classical frame operators or weakly-defined multipliers in the bounded context. In general some properties of them, such as the invertibility and the resolvent set, are related to properties of the sesquilinear forms.  \\
	As an upshot of this approach new features of sequences (or pairs of sequences) which are semi-frames (or reproducing pairs) are obtained.
\end{minipage}

\vspace*{.5cm}

\begin{minipage}{11.8cm}
\end{minipage}

\vspace*{.5cm}

\begin{minipage}{11.8cm}
\end{minipage}

\vspace*{1cm}
\normalsize

\section{Introduction}

\no Let $\H$ be a Hilbert space with inner product $\pint$ and norm $\nor$. Given two sequences $\xi:=\{\xi_n\}$ and $\eta:=\{\eta_n\}$ of elements of $\H$, a sesquilinear form on a suitable domain $\D_1\times \D_2$ can be defined as
$$
\O_{\xi,\eta}(f,g)=\sum_{n=1}^{\infty} \pin{f}{\xi_n}\pin{\eta_n}{g}, \qquad f\in \D_1,g\in\D_2.
$$

\no Obviously, a particular case appears when both subspaces coincide with $\H$. Assuming that $\xi=\eta$, $\O_{\xi,\xi}$ is defined on $\H\times \H$ and it is bounded if and only if $\xi$ is a Bessel sequence. 
The case with different sequences includes the notion of {\it reproducing pair}, that was introduced in \cite{Speck_Bal_15,Speck_Bal_16} and studied also in \cite{AST,Anto_Tp}. In the discrete formulation, two sequences $\xi,\eta$ constitute a reproducing pair of $\H$ if $\O_{\xi,\eta}$ is defined on $\H\times \H$, is bounded and the operator $T_{\xi,\eta}$ associated to $\O_{\xi,\eta}$, i.e., 
$$
\O_{\xi,\eta}(f,g)=\pin{T_{\xi,\eta} f}{g}, \qquad \forall f,g\in \H,
$$
is invertible with bounded inverse. This leads to the following formulas, in {\it weak sense}, to express an element $f\in \H$
$$
f=\sum_{n=1}^{\infty}\pin{f}{{T_{\xi,\eta}^*}^{-1}\xi_n}\eta_n=\sum_{n=1}^{\infty}\pin{f}{T_{\xi,\eta}^{-1}\eta_n}\xi_n.
$$
In the case where $\xi$ is a Bessel sequence and $\xi=\eta$, $T_{\xi,\xi}$ is given by $T_{\xi,\xi} f= \sum_{n=1}^\infty \pin{f}{\xi_n}\xi_n$ in strong sense. If moreover $\xi$ is a frame, then $T_{\xi,\xi}$ is bijective and is called, as known, the frame operator of $\xi$.

In addition to reproducing pairs, other generalizations of the notion of frame have been introduced; for instance semi-frames \cite{Classif,Anto_Bal_1,Anto_Bal,Casazza_lower}. \\
For two general sequences $\xi,\eta$ the form $\O_{\xi,\eta}$ might be unbounded. Nevertheless, for unbounded sesquilinear forms $\O$ on a domain $\D_1\times \D_2$ several representation theorems through operators $T$,
$$
\O(f,g)=\pin{Tf}{g}, \qquad\forall f\in \dom(T)\subseteq \D_1, g\in \D_2,
$$
have been formulated (see \cite{RC_CT} where the notion of {\it solvable} form is developed). Our aim is to apply these theorems in the context of the sesquilinear forms associated to one or two sequences and study the associated operators. An analogous approach for one sequence was applied in \cite{Bag_sesq} in the framework of generalized Riesz systems. \\
We will consider in particular the following type of forms: {\it closed nonnegative} (studied by Kato \cite{Kato}), {\it $\lambda$-closed} where $\lambda \in \C$ (studied by McIntosh \cite{McIntosh70}) and solvable forms. The operators associated to these forms are closed, and also densely defined provided that both $\D_1$ and $\D_2$ are dense. In the first case they are self-adjoint and positive. In the second case  
their resolvent sets are always not empty.

This paper is structured as follows. We recall some notions on sesquilinear forms and on sequences in Section \ref{sec:prel}. Here we state the representation theorem for solvable forms and, in particular, for $\lambda$-closed forms. 

In Section \ref{sec:1seq} we start with defining the sesquilinear form $\O_\xi:=\O_{\xi,\xi}$ associated to a sequence $\xi$ on  
$\dom(\xi):=\left \{f\in \H: \sum_{n=1}^{\infty}|\pin{f}{\xi_n}|^2<\infty \right \}$, that is the greatest possible domain. This form is nonnegative and closed. Therefore, if it is densely defined, by Kato's theorems, it is represented by a nonnegative self-adjoint operator $T_\xi$, that is exactly $C_\xi^*C_\xi=|C_\xi|^2$, where $C_\xi$ is the analysis operator of $\xi$. Clearly, $T_\xi$ is an extension of the operator 
$
S_\xi f = \sum_{n=1}^\infty \pin{f}{\xi_n}\xi_n,
$
defined for $f\in \H$ such that the series converges in strong sense (called the 'frame-operator' of $\xi$ in some papers like \cite{Classif}). Differently from the bounded case (i.e., when $\xi$ is a Bessel sequence), $T_\xi$ may be different from $S_\xi$ (see Example \ref{count_exm_TS_xi}). We also give some characterization of $\xi$ in terms of $\O_\xi$ and $T_\xi$ in Propositions \ref{car_form_1_seq} and \ref{car_seq_T_xi}, respectively. \\
Furthermore, we consider also another sesquilinear form for a sequence $\xi$. More precisely, with $\Theta_\xi(\{c_n\},\{d_n\})=\sum_{i,j\in \N} c_i\ol{d_j}\pin{\xi_i}{\xi_j}$ one can define a nonnegative form $\Theta_\xi$  on $\dom(D_\xi)\times\dom(D_\xi)$, where $D_\xi$ denotes the synthesis operator of $\xi$. In contrast with $\Omega_\xi$, $\Theta_\xi$ is always densely defined; moreover, $\Theta_\xi$ is closable if and only if $\dom(\xi)$ is dense.

Section \ref{sec:2seq} deals with sesquilinear forms associated to two sequences. One of the main problems is the domain on which $\O_{\xi,\eta}$ can be defined. If the sequences are different then typically there does not exist the greatest domain. First, we analyze the bounded case: the operator that represents the form acts as $T_{\xi,\eta} f = \sum_{n=1}^\infty \pin{f}{\xi_n}\eta_n$ in  weak sense for $f\in \H$.  \\
In the general case, we define the form $\O_{\xi,\eta}$ on $\dom(\xi)\times \dom(\eta)$. Under this assumption we find in Theorem \ref{th_cns_0clos} that $\O_{\xi,\eta}$ is $0$-closed if and only if $\xi,\eta$ are lower semi-frames and $R(C_\xi) \dotplus R(C_\eta)^\perp= l_2$ (or equivalently $R(C_\eta) \dotplus R(C_\xi)^\perp= l_2$) holds, where $\dotplus$ stands for the direct sum of subspaces. If $\dom(\eta)$ is dense, the operator associated to this form is $C_\eta ^* C_\xi$, and it is invertible with bounded inverse if and only if the form is $0$-closed. As a consequence we recover reconstruction formulas in weak sense. Finally, the possibility of choosing different domains of $\O_{\xi,\eta}$ and existence of maximal domains is discussed.  

In Section \ref{sec:exm} we apply the obtained results in examples involving weighted Riesz basis or weighted Bessel sequences. 
In the last section we write two sequences $\xi,\eta$ as $\xi_n=V e_n$ and $\eta_n=Z e_n$ for a fixed orthonormal basis $\{e_n\}$ and for some operators $V,Z$. We analyze the relations between $V,Z$ and $\O_{\xi,\eta}$.

\section{Preliminaries}
\label{sec:prel}

Let $\H$ be a Hilbert space with inner product $\pint$ and norm $\nor$.  
We denote by $\dom(T),N(T),R(T),\rho(T)$ the domain, kernel, the range and the resolvent set of an operator $T$ from $\H_1$ into $\H_2$, respectively, where $\H_1,\H_2$ are Hilbert spaces. We indicate by $I$ the identity operator and by $\B(\H)$ the set of bounded operators, everywhere defined on $\H$. An operator $T$ is called {\it semi-bounded} if there exists $c>0$ such that $\n{Tf}\geq c \n{f}$ for all $f\in \dom(T)$.\\
For a complex sequence $\alpha=\{\alpha_n\}$ we set $\alpha^2:=\{\alpha_n^2\}$. 
Moreover $l_2(\alpha)$ stands for the Hilbert space of complex sequences $\{c_n\}$ satisfying $\sum_{n=1}^\infty |\alpha_n||c_n|^2 <\infty$. The norm of $\{c_n\}\in l_2(\alpha)$ is given by $(\sum_{n=1}^\infty |\alpha_n||c_n|^2)^\mez$. For simplicity, we use the classic notation $l_2$ for the space $l_2(\{1\})$. 

\subsection{Sesquilinear forms}

Basic notions on sesquilinear forms can be found in \cite[Ch. VI]{Kato}. We recall that if $\D_1,\D_2$ are subspaces of $\H$ and $\O$ is a sesquilinear form on $\D_1\times \D_2$ then, the {\it adjoint} $\O^*$ of $\O$ is defined on $\D_2\times \D_1$ as
$$
\O^*(\phi,\psi)=\ol{\O(\psi,\phi)}, \qquad  \phi\in \D_2, \psi\in \D_1.
$$
In the case where $\D:=\D_1=\D_2$, 
$\O$ on $\D\times \D$ is called 
\begin{itemize}
	\item {\it symmetric} if $\O=\O^*$;
	\item {\it semi-bounded} with {\it lower bound} $\gamma \in \R$ if $\O(f,f)\geq \gamma \n{f}^2$ for all $f \in \D$.
	\item {\it nonnegative} if $\O(f,f)\geq 0$ for all $f \in \D$ (in this case we use the symbol $\O\geq 0$).
\end{itemize}

\no A sesquilinear form $\O$ on $\D_1\times \D_2$ is said to be
\begin{itemize}
	\item {\it densely defined} if $\D_1, \D_2$ are dense (in $\H$);
	\item {\it bounded on $\H$} if for some $C> 0$, 
	$|\O(\phi,\psi)|\leq C \n{\phi}\n{\psi}$ for all $\phi\in \D_1,\psi \in \D_2$. If $\D_1=\D_2=\H$, then the {\it norm} of $\O$ is 
	$$\n{\O}:=\sup_{f,g\in \H\backslash{\{0\}}} \frac{|\O(f,g)|}{\n{f}\n{g}}.$$
\end{itemize} 

\no We denote by $N(\O):=\{f\in\D_1: \O(f,g)=0, \forall g\in\D_2\}$. Then $N(\O^*)=\{g\in\D_2: \O(f,g)=0, \forall f\in\D_1\}$. Moreover, we put $\iota(f,g):=\pin{f}{g}$ for all $f,g\in \H$.\\

Let $\O$ be a sesquilinear form on $\D_1\times \D_2$ with $\D_2$ dense. The well-defined operator $T$ on 
\begin{align}
\label{eq_dom_intro1}
\dom(T)=\{f \in \D_1:\exists h \in \H, \O(f,g)=\pin{h}{g}, \forall g \in \D_2\}
\end{align}
given by $Tf=h$, for all $f\in \dom(T)$ and $h$ as in (\ref{eq_dom_intro1}), is called the {\it operator associated} to $\O$. This operator is the greatest one satisfying $\dom(T)\subseteq \D_1$ and for which the representation $\O(f,g)=\pin{Tf}{g}$ with $f\in \dom(T),g\in \D_2$ holds. \\
It is not densely defined nor closed, in general. However, under some further conditions that we are going to introduce, these two properties are obtained.

\begin{defin}
\label{def_qclo}
A sesquilinear form $\O$ on $\D_1\times \D_2$ is called {\it q-closed} if there exist Hilbertian norms $\nor_1$ on $\D_1$ and $\nor_2$ on $\D_2$ such that
\begin{enumerate}[label={(\roman*)}]
	\item the embeddings $\D_1[\nor_1]\to \H$ and $\D_2[\nor_2]\to \H$ are continuous;
	\item there exists $\beta >0$ such that $|\O(f,g)|\leq \beta\n{f}_1\n{g}_2$, for all  $f\in \D_1,g \in \D_2$. 
\end{enumerate}
\end{defin}

\no The next definition is based by \cite{McIntosh70}, but here we do not assume that the subspaces are dense.

\begin{defin}
Let $\lambda \in \C$. A q-closed sesquilinear form $\O$ on $\D_1\times \D_2$ is called {\it $\lambda$-closed} if 
\begin{enumerate}[label={(\roman*)}]
	\item if $(\O-\lambda \iota)(f,g)=0$ for all $g \in \D_2$, then $f=0$, i.e., $N(\O-\lambda \iota)=\{0\}$;
	\item for every anti-linear continuous functional $\L$ on $\D_2[\nor_2]$ there exists $f\in \D_1$ such that $\L(g)=(\O-\lambda \iota)(f,g)$.
\end{enumerate}
\end{defin}

\begin{defin}
\label{def_solvable}
A q-closed sesquilinear form $\O$ on $\D_1\times \D_2$ is called {\it solvable} if there exists a bounded sesquilinear form $\Up$ on $\H\times \H$ such that $\O+\Up$ is $0$-closed. 
\end{defin}

Solvable (and in particular $\lambda$-closed) sesquilinear forms are generalizations of Kato's closed sectorial forms \cite[Ch. VI]{Kato}. In particular, a semi-bounded form $\O$ with lower bound $\gamma$ on $\D\times \D$ is {\it closed} if $\D$ is complete when endowed with the inner product $(\O-\gamma\iota) (f,g)$, $f,g\in \D$.

The original definition of solvable forms goes back to \cite{Tp_DB,RC_CT,Second} where $\D_1=\D_2$ (and dense) is always assumed. Note also that in Definition \ref{def_solvable} we have preferred to use the simple terminology 'solvable' instead of 'solvable with respect to an inner product' as in 	\cite{RC_CT,Second}. \\
However, the next results can be easily adapted from \cite[Theorems 4.6, 4.11]{RC_CT} and \cite[Theorem 2.5]{Second} (see also \cite[Proposition 2.1]{McIntosh70} and \cite[Theorem 7.2]{RC_CT}).

\begin{theo}
	\label{th_rapp_risol}
	Let $\O$ be a solvable sesquilinear form on $\D_1\times \D_2$ with $\D_2$ dense in $\H$ and $T$ its associated operator. The following statements hold.
	\begin{enumerate}[label=\emph{(\roman*)}]
		\item $\dom(T)$ is dense in $\D_1[\nor_1]$. If $\D_1$ is dense in $\H$, then also $\dom(T)$ is dense in $\H$. 
		\item Let $\Up$ be a bounded sesquilinear form and $B\in \B(\H)$ the bounded operator associated to $\Up$. Then $\O+\Up$ is $0$-closed if and only if $0\in \rho(T+B)$.
		In particular, $\O$ is $\lambda$-closed with $\lambda \in \C$ if and only if $\lambda \in \rho(T)$.
		\item $T$ is closed.
		\item $\O^*$ is solvable. More precisely, if $\Up$ is a bounded sesquilinear form, then $\O+\Up$ is $0$-closed if and only if $\O^*+\Up^*$ is $0$-closed.
		\item Assume that $\D_1$ is dense in $\H$. The operator associated to $\O^*$ is $T^*$.
		\item If $\D_1=\D_2$, then $\O$ is symmetric if and only if $T$ is self-adjoint.
	\end{enumerate}
\end{theo}	

In the next sections we will need the following criterion to establish if a given q-closed form is also $0$-closed. The proof is similar to the one of Lemma 5.1 of \cite{RC_CT}.

\begin{lem}
	\label{lem_cns_0clos}
	Let $\O$ be a q-closed sesquilinear form on $\D_1\times \D_2$. Let $\nor_1$ and $\nor_2$ be the norms on $\D_1$ and $\D_2$ according to Definition \ref{def_qclo}, respectively. The following statements are equivalent.
	\begin{enumerate}[label=\emph{(\roman*)}]
		\item $\O$ is $0$-closed;
		\item  $N(\O)=\{0\}$ and there exists $c_2>0$ such that
		$$
		c_2\n{g}_2\leq \sup_{\n{f}_1=1} |\O(f,g)| \qquad \forall g \in \D_2;
		$$
		\item $N(\O^*)=\{0\}$ and there exists $c_1>0$ such that
		$$
		c_1\n{f}_1\leq \sup_{\n{g}_2=1} |\O(f,g)| \qquad \forall f \in \D_1;
		$$
		\item there exist $c_1,c_2>0$ such that
		$$
		c_1\n{f}_1\leq \sup_{\n{g}_2=1} |\O(f,g)| \qquad \forall f \in \D_1,
		$$
		$$
		c_2\n{g}_2\leq \sup_{\n{f}_1=1} |\O(f,g)| \qquad \forall g \in \D_2.
		$$
	\end{enumerate}
\end{lem}

\subsection{Sequences}

For a sequence $\xi=\{\xi_n\}$ of $\H$ we denote by 
$$\dom(\xi):=\left \{f\in \H: \sum_{n=1}^{\infty}|\pin{f}{\xi_n}|^2<\infty \right \}.$$
If an element $h\in \H$ is the strong limit of $\sum_{n=1}^k \xi_n$, then we write $h=\sum_{n=1}^\infty \xi_n$; while if it is the weak limit of $\sum_{n=1}^k \xi_n$, i.e., $\pin{h}{g}=\sum_{n=1}^\infty \pin{\xi_n}{g}$ for all $g\in \H$, we write $h=(w)\sum_{n=1}^\infty \xi_n$.

We will use the abbreviation ONB to mean orthonormal basis. For the 
following notions we refer to \cite{Classif,Anto_Bal,Anto_Bal_1,Chris}.
A sequence $\xi$ is an {\it Bessel sequence} of $\H$ with upper bound $B>0$ if
\begin{equation}
\label{up_frame}
\sum_{n=1}^\infty |\pin{f}{\xi_n}|^2\leq B\n{f}^2, \qquad \forall f\in \H.
\end{equation}
In particular, if in (\ref{up_frame}) the left hand side is zero only for $f=0$, then $\xi$ is called {\it upper semi-frame}. \\
A sequence $\xi$ is a {\it lower semi-frame} of $\H$ with lower bound $A>0$ if 
$$
A\n{f}^2 \leq \sum_{n=1}^\infty |\pin{f}{\xi_n}|^2, \qquad \forall f\in \H.
$$ 
Note that the series on the right may diverge for some $f\in \H$. More precisely (see \cite[Proposition 4.1]{Classif}), the series is convergent for all $f\in \H$ if and only if $\xi$ is also a {\it frame}, i.e., there exists $A,B>0$ such that
$$
A\n{f}^2 \leq \sum_{n=1}^\infty |\pin{f}{\xi_n}|^2\leq B\n{f}^2, \qquad \forall f\in \H.
$$ 
A {\it Riesz basis} $\xi$  is a sequence satisfying for some $A,B>0$
$$
A \sum_{n=1}^\infty |c_n|^2 \leq \left \| \sum_{n=1}^\infty c_n \xi_n \right \|^2 \leq B\sum_{n=1}^\infty |c_n|^2, \qquad\forall \{c_n\}\in l_2.
$$
Instead, a sequence $\xi$ satisfying only the first inequality above, for $\{c_n\}\in l_2$ such that $\sum_{n=1}^\infty c_n \xi_n$ exists, is called {\it Riesz-Fischer sequence}.

Two sequences $\xi=\{\xi_n\}$ and $\eta=\{\eta_n\}$ are said to be {\it biorthogonal} if $\pin{\xi_n}{\eta_m}=\delta_{n,m}$, where $\delta_{n,m}$ is the Kronecker symbol.

There are three operators that are classically associated to a sequence $\xi$. 
The {\it analysis operator} $C_\xi :\dom(C_\xi )\subseteq \H\to l_2$ is given by $\dom(C_\xi )=\dom(\xi)$  and $C_\xi f=\{\pin{f}{\xi_n}\}$, for all $f\in \dom(C_\xi)$. The {\it synthesis operator} $D_\xi:\dom(D_\xi)\subseteq l_2 \to \H$ is given by 
$$\dom(D_\xi):=\left \{ \{c_n\}\in l_2:\sum_{n=1}^\infty c_n \xi_n \text{ exists in } \H\right \}$$ 
and $D_\xi\{c_n\}=\sum_{n=1}^\infty c_n \xi_n$, for $ \{c_n\}\in \dom(D_\xi)$. Finally let $S_\xi$ be the operator with  
$$\dom(S_\xi):=\left \{f\in \H: \sum_{n=1}^\infty \pin{f}{\xi_n}\xi_n \text{ exists in } \H\right \}$$
and $S_\xi f =\sum_{n=1}^\infty \pin{f}{\xi_n}\xi_n$, for $f\in \dom(S_\xi)$. The basic properties of these operators are listed below.

\begin{pro}[{\cite[Prop. 3.3]{Classif}}] The following statements hold.
	\label{pro_oper_1}
	\begin{enumerate}[label=\emph{(\roman*)}]
		\item $C_\xi=D_\xi^*$ and $C_\xi$ is closed.
		\item If $C_\xi$ is densely defined, then $D_\xi \subseteq C_\xi^*$ and $D_\xi$ is closable.
		\item $S_\xi=D_\xi C_\xi$.
	\end{enumerate}
\end{pro}

\no If $C_\xi$ is densely defined then it may happen that $D_\xi = C_\xi^*$ (for instance if $\xi$ is a frame) or $D_\xi \neq C_\xi^*$ (like in the last example of \cite{Ole}). More precisely, the operator $C_\xi^*$ has domain
\begin{align*}
\dom(C_\xi^*)
&=\{\{c_n\}\in l_2:f\to\pin{C_\xi f}{\{c_n\}}_{l_2} \text{ is bounded on }\dom(\xi)\}\\
&=\left\{\{c_n\}\in l_2:f\to\sum_{n=1}^\infty  \pin{f}{\xi_n}\ol{c_n} \text{ is bounded on }\dom(\xi)\right \}.
\end{align*}

\section{Sesquilinear forms associated to a sequence}
\label{sec:1seq}

\no Now, consider the nonnegative sesquilinear form
$$
\O_\xi(f,g)=\sum_{n=1}^\infty \pin{f}{\xi_n}\pin{\xi_n}{g}.
$$
The largest domain $\dom(\O_\xi)$ on which $\O_\xi$ is defined is exactly $\dom(\xi)$. Then, clearly,
\begin{equation}
\label{O_xi_max}
\O_\xi(f,g)=\pin{C_\xi f}{C_\xi g}_2, \qquad\forall f,g\in \dom(\xi).
\end{equation}

\no Since $C_\xi$ is a closed operator, $\O_\xi$ is a closed nonnegative form. Basing on Proposition 4.1 of \cite{Classif} we can state also some characterizations of $\xi$ in terms of $\O_\xi$.

\begin{pro}
\label{car_form_1_seq}
		Let $\xi$ be a sequence of $\H$. The following statements hold.
	\begin{enumerate}[label=\emph{(\roman*)}]
		\item $\xi$ is complete if and only if $N(\O_\xi)=\{0\}$.
		\item $\xi$ is a Bessel sequence if and only if $\dom(\O_\xi)=\H$. 
		\item $\xi$ is a Bessel sequence with upper bound $B$ if and only if $\dom(\O_\xi)=\H$,  $\O_\xi$ is bounded and $\n{\O_\xi}\leq B$.
		\item $\xi$ is an upper semi-frame if and only if $\dom(\O_\xi)=\H$ and $N(\O_\xi)=\{0\}$.
		\item $\xi$ is a lower semi-frame with lower bound $A$ if and only if $\O_\xi$ is semi-bounded with lower bound $A$.
		\item $\xi$ is a frame if and only if $\dom(\O_\xi)=\H$ and $\O_\xi$ is semi-bounded with positive lower bound.
		\item $\xi$ is a frame if and only if $\dom(\O_\xi)=\H$, $N(\O_\xi)=\{0\}$ and for every $h\in \H$ there exists $f\in \H$ such that $\O_\xi(f,g)=\pin{h}{g}$ for all $g\in \H$.
		\item If $\xi$ is a Riesz-Fischer sequence, then for every $f',g'\in \H$ there exist $f,g\in \dom(\O_\xi)$ such that $\O_\xi(f,g)=\pin{f'}{g'}$.
		\item If $\xi$ is a Riesz basis, then $\dom(\O_\xi)=\H$ and for every $f',g'\in \H$ there exist $f,g\in \H$ such that $\O_\xi(f,g)=\pin{f'}{g'}$.
	\end{enumerate} 
\end{pro}
%

\no Suppose that $\O_\xi$ is densely defined, i.e., $\dom(\xi)$ is dense (a sufficient condition for this property is given by \cite[Lemma 3.1]{Classif}). By Kato's first representation theorem \cite[Theorem VI.2.1]{Kato}, the operator $T_\xi$ associated to $\O_\xi$ is positive and self-adjoint. Moreover, by Kato's second representation theorem \cite[Theorem VI.2.23]{Kato} we have also that $\dom(\O_\xi)=\dom(T_\xi^\mez)$ and
\begin{align}
\label{op_ass_xi}
\O_\xi(f,g)&=\pin{T_\xi^\mez f}{T_\xi^\mez g}, \qquad \forall f, g\in \dom(\O_\xi). \nonumber
\end{align}
By (\ref{O_xi_max}) one can easily see that $T_\xi=C_\xi^*C_\xi=|C_\xi|^2$. Thus the domain of $T_\xi$ is 
$$
\dom(T_\xi)=\left \{f\in \H:g\mapsto \sum_{n=1}^\infty \pin{f}{\xi_n}\pin{\xi_n}{g} \text{ is bounded on } \dom(\xi) \right\} =\dom(|C_\xi|^2).
$$


\no Then $T_\xi$ is an extension of $S_\xi$.
It is a well-known fact that if $\xi$ is a Bessel sequence, then the operator associated to $\O_\xi$ is $S_\xi$, i.e., $T_\xi=S_\xi$. The following example shows however that, in the general case, $T_\xi$ does not always coincide with $S_\xi$.

\begin{exm}
	\label{count_exm_TS_xi}
	Let $\{e_n\}$ be an ONB of $\H$. For $f\in \H$ we denote by $f_n$ the coefficient of $f$ with respect to that basis.\\
	Let us define $\xi_1=e_1$ and $ \xi_n=n(e_n-e_{n-1})$ for $n\geq 2$. Then
$\dom(\xi)=\{f\in \H: \sum_{n=1}^\infty n^2|f_n-f_{n-1}|^2< \infty\}$. 
	For $k>1$ and $\{c_n\}\in l_2$
	\begin{align*}
	\sum_{n=1}^k c_n\xi_n 
	&= \sum_{n=1}^{k-1} (n c_n -(n+1)c_{n+1})e_n+kc_k e_k.
	\end{align*}
	Let $f\in \H$ be such that $f_n=\frac{1}{n}$, for $n\geq 1$. 
	Since
	\begin{align*}
	\sum_{n=1}^k \pin{f}{\xi_n}\xi_n &=\sum_{n=1}^{k-1} (n\pin{f}{\xi_n} -(n+1)\pin{f}{\xi_{n+1}})e_n +k\pin{f}{\xi_k} e_k  \\
	&=-\sum_{n=1}^{k-1}  \frac{1}{n(n-1)} e_n -\frac{k}{k-1} e_k,
	\end{align*}	
	$f\not \in \dom (S_\xi)$, but the functional $g\mapsto \sum_{n=1}^\infty c_n \pin{\xi_n}{g}$ is bounded for $g\in \dom(\xi)$, i.e., $f\in \dom(T_\xi)$. 
	
\end{exm}

Taking a sequence $\xi$, it is easy to define a new sequence $\xi'$ which is a lower semi-frame. Indeed, one can take $\{\xi_n'\}=\{e_1,\xi_1,\dots, e_n,\xi_n, \dots\}$ where $\{e_n\}$ is a ONB. Clearly, $\sum_{n=1}^{\infty} |\pin{f}{\xi_n'}|^2=\sum_{n=1}^{\infty} |\pin{f}{\xi_n}|^2+\n{f}^2$. Hence $\dom(C_{\xi'})=\dom(C_\xi)$ and $\n{C_{\xi'} f}\geq \n{f}$. Moreover, we have also $\dom(D_{\xi'})=\dom(D_{\xi})$, $\dom(S_{\xi})=\dom(S_{\xi'})$ e $\dom(T_{\xi})=\dom(T_{\xi'})$.

\begin{exm}
	The operator $T_\xi$ and $S_\xi$ may be different even if $\xi$ is a lower semi-frame. Indeed with the notation of the previous example, let  $\xi'=\{e_1,\xi_1,\\ \dots, e_n,\xi_n, \dots\}.$ 
	This sequence is then a lower semi-frame with $\dom(\xi)$ dense, but $T_{\xi'}$ is a proper extension of $S_{\xi'}$.
\end{exm}

\no The equality $T_\xi=S_\xi$ holds if $D_\xi$ is closed (indeed we have $D_\xi=C_\xi^*$ in that case).  We give now another characterization of $\xi$ which involves now the operator $T_\xi$.

\begin{pro}
	\label{car_seq_T_xi}
	Let $\xi$ be a sequence of $\H$ with $\dom(\xi)$ dense. The following statements hold.
	\begin{enumerate}[label=\emph{(\roman*)},ref=(\roman*)]
		\item \label{pro_comp} $\xi$ is complete if and only if $T_\xi$ is injective.
		\item \label{pro_Bess} $\xi$ is a Bessel sequence if and only if $\dom(T_\xi)=\H$ if and only if $T_\xi\in \B(\H)$.
		\item \label{pro_Bess2} $\xi$ is a Bessel sequence with bound $B$ if and only if $T_\xi\in \B(\H)$ and $\n{T_\xi}\leq B$.
		\item \label{pro_upper} $\xi$ is an upper semi-frame if and only if $T_\xi\in \B(\H)$ and $T_\xi$ is injective.
		\item \label{pro_lower} $\xi$ is a lower semi-frame with bound $A$ if and only if $0 \in \rho(T_\xi)$ and $\n{T_\xi^{-1}}\leq A$. 
		\item \label{pro_frame1} $\xi$ is a frame if and only if $\dom(T_\xi)=\H$ and $T_\xi$ is bijective.
		\item \label{pro_frame2} $\xi$ is a frame if and only if $\dom(T_\xi)=\H$ and $T_\xi$ is surjective.
		\item \label{pro_Riesz} $\xi$ is a Riesz basis if and only if $\dom(T_\xi)=\H$, $T_\xi$ is injective and $\{T_\xi^{-1}\xi_n\}$ is biorthogonal to $\xi$. 
	\end{enumerate} 
\end{pro}
\begin{proof}
	Point \ref{pro_comp} is clear. Points \ref{pro_Bess}, \ref{pro_Bess2}, \ref{pro_upper}, \ref{pro_frame1}, \ref{pro_frame2} and \ref{pro_Riesz} follows by \cite[Prop. 4.3]{Classif}. To prove \ref{pro_lower}, note that $\xi$ is a lower semi-frame if and only if $\O_\xi$ is semi-bounded with positive lower bound if and only if $0 \in \rho(T_\xi)$. Moreover, if $A>0$, $\n{T_\xi^\mez f}^2=\O_\xi(f,f)\geq A\n{f}^2$ if and only if $\n{T_\xi^{-1}}\leq A$.
\end{proof}

Assume that $\xi$ is a lower semi-frame with $\dom(\xi)$ dense. Thus $0\in \rho(T_{\xi})$. If $S_\xi=T_\xi$ then we obtain the following {\it reconstruction formula} in strong sense
\begin{align*}
f&=T_\xi T_\xi^{-1}f=\sum_{n=1}^{\infty}\pin{f}{T_\xi^{-1}\xi_n}\xi_n, \qquad \forall f\in \H.
\end{align*}
If $S_\xi \subsetneq T_\xi$, then we have only a formula in strong sense on $R(S_\xi)$
\begin{align*}
f&=\sum_{n=1}^{\infty}\pin{S_\xi^{-1}f}{\xi_n}\xi_n,  \qquad \forall f\in R(S_{\xi}).
\end{align*}
In general, the reconstruction formula in weak sense (\ref{weak_rec_1seq(a)}) below holds. Let $h\in \H$. For all $g\in \dom(\xi)$
\begin{equation}
\label{weak_rec_1seq(a)}
\pin{h}{g}=\pin{T_\xi T_\xi ^{-1}h}{g}=\sum_{n=1}^\infty \pin{T_\xi^{-1}h}{\xi_n}\pin{\xi_n}{g}=\sum_{n=1}^\infty \pin{h}{T_\xi^{-1}\xi_n}\pin{\xi_n}{g}.
\end{equation}
Note that $\{T_\xi^{-1}g_n\}$ is a Bessel sequence. Indeed, for every $f\in \H$, $T_\xi^{-1}f\in \D(C)$ and
\begin{align*}
\sum_{n=1}^\infty |\pin{f}{T_\xi^{-1}g_n}|^2=\|C_\xi T_\xi^{-1}f\|^2_2=\|U|C_\xi|^{-1}f\|^2\leq \|U|C_\xi|^{-1}\|^2\|f\|^2
\end{align*}
taking into account the polar decomposition of $C_\xi$, $C_\xi=U|C_\xi|$ with partial isometry $U$ and modulus $|C_\xi|$. Thus, we obtain from  \eqref{weak_rec_1seq(a)} the following reconstruction in strong sense
\begin{equation}
\label{weak_rec_1seq(b)}
g=\sum_{n=1}^\infty \pin{g}{\xi_n} T_\xi^{-1}\xi_n, \qquad \forall g\in \dom(\xi).
\end{equation}
Actually, a formula like \eqref{weak_rec_1seq(b)} involving a lower semi-frame $\xi$ and a Bessel sequence holds even if $\dom(\xi)$ is not dense (see \cite[Proposition 3.4]{Casazza_lower}). However, $\dom(\xi)$ must be dense to define $T_\xi$ and, following the case with frames, we can call $\{T_\xi^{-1}\xi_n\}$ in \eqref{weak_rec_1seq(b)} the {\it canonical dual} of the lower semi-frame $\xi$.\\

%

\no Another form that can be defined starting from a sequence $\xi=\{\xi_n\}$ of $\H$ is 
$$
\Theta_\xi(\{c_n\},\{d_n\})=\sum_{i,j\in \N} c_i\ol{d_j}\pin{\xi_i}{\xi_j}.
$$
This form is well-defined on $\dom(D_\xi)\times \dom(D_\xi)$. More precisely, if $\{c_n\},\{d_n\} \in \dom(D_\xi)$ then $\Theta_\xi(\{c_n\},\{d_n\})=\pin{D_\xi\{c_n\}}{D_\xi\{d_n\}}$. Basing on classic properties of closed nonnegative forms and on \cite[Prop. 4.2]{Classif}, we can formulate the next results, where we consider $\Theta_\xi$ always on the domain $\dom(\Theta_\xi):=\dom(D_\xi)$. 

\begin{pro}
	Let $\xi$ be a sequence of $\H$. The following statements hold.
	\begin{enumerate}[label=\emph{(\roman*)}]
		\item $\Theta_\xi$ is nonnegative and densely defined.
		\item $\Theta_\xi$ is closable if and only if $\dom(\xi)$ is dense.
		\item $\Theta_\xi$ is closed if and only if $D_\xi$ is closed if and only if $\dom(\xi)$ is dense and $D_\xi=C_\xi^*$. 
		\item If $\dom(\xi)$ is dense, then the closure $\ol{\Theta_\xi}$ of $\Theta_\xi$ is the sesquilinear form on $\dom(\ol{\Theta_\xi})=\dom(C_\xi^*)$ given by
		\begin{align*}
		\ol{\Theta_\xi}(\{c_n\},\{d_n\})&=\pin{C_\xi^*\{c_n\}}{C_\xi^*\{d_n\}}, 
		\qquad \forall \{c_n\},\{d_n\}\in \dom(C_\xi^*),
		\end{align*}
		and the operator associated to $\ol{\Theta_\xi}$ is $C_\xi C_\xi^*=:|C_\xi^*|^2$.
		\item $\xi$ is a Bessel sequence if and only if $\dom(\Theta_\xi)=l_2$.
		\item $\xi$ is a Riesz-Fischer sequence if and only if $\Theta_\xi$ is semi-bounded with positive lower bound.
		\item $\xi$ is a Riesz basis if and only if $\Theta_\xi$ is bounded and semi-bounded with positive lower bound.
		\item If $\xi$ is a frame then $\dom(\Theta_\xi)=l_2$ and for every $f',g'\in \H$ there exists $f,g\in \dom(\Theta_\xi)$ such that $\Theta_\xi(f,g)=\pin{f'}{g'}$.
	\end{enumerate}
\end{pro}

\subsection{Lower semi-frames as frames in a different Hilbert space}

\no We conclude this section noting that lower semi-frames are frames in some Hilbert space continuously embedded into $\H$.

\begin{pro}
	\label{pro_lower_frame}
	Let $\xi$ be a sequence of $\H$. The following statements are equivalent.
	\begin{enumerate}[label=\emph{(\roman*)}]
		\item $\xi$ is a lower semi-frame;
		\item there exists a inner product $\pint_+$ inducing a norm $\nor_+$ on $\dom(\xi)$ such that $\dom(\xi)[\nor_+]$ is complete and, for some $\alpha,A,B>0$, 
		$$\alpha \n{f} \leq \n{f}_+ \qquad \text{and}$$ 
		$$
		A\n{f}_+^2 \leq \sum_{n=1}^{\infty}|\pin{f}{\xi_n}|^2 \leq B\n{f}_+^2, \qquad \forall f\in \dom(\xi);
		$$
		\item for all inner product $\pint_+$ inducing a norm $\nor_+$ on $\dom(\xi)$ such that $\dom(\xi)[\nor_+]$ is complete and $\alpha \n{f} \leq \n{f}_+$ for some $\alpha >0$, there exist $A,B>0$, such that
		\begin{equation}
		\label{eq_quasi_frame}
			A\n{f}_+^2 \leq \sum_{n=1}^{\infty}|\pin{f}{\xi_n}|^2 \leq B\n{f}_+^2, \qquad \forall f\in \dom(\xi).
		\end{equation}
	\end{enumerate}
\end{pro}
\begin{proof}
	(i)$\Rightarrow$(ii) It is sufficient to take $\n{f}_+=(\sum_{n=1}^{\infty}|\pin{f}{\xi_n}|^2)^\mez$ for $f\in \dom(\xi)$.\\
	(ii)$\Rightarrow$(iii) By the closed graph theorem all norm which turn $\dom(\xi)$ into a Hilbert space continuously embedded into $\H$ are equivalent. \\
	(iii)$\Rightarrow$(i) The assertion follows easily since a norm $\nor_+$ satisfying (\ref{eq_quasi_frame}) is equivalent to the norm $f\to \left (\sum_{n=1}^{\infty}|\pin{f}{\xi_n}|^2 \right )^\mez$.
\end{proof}

\no Let $\xi$ be a lower semi-frame and $\pint_+$ be a inner product that makes $\dom(\xi)$ into a complete space (when we write $\dom(\xi)$ here we mean that it is endowed with this inner product). For every $n\in \N$ and $f\in \dom(\xi)$, $f \mapsto \pin{f}{\xi_n}$ defines a bounded functional on $\dom(\xi)$. By Riesz's Lemma there exists a sequence $\xi'=\{\xi_n'\}$ in $\dom(\xi)$ such that
\begin{equation}
\label{Riesz_qrp}
\pin{f}{\xi_n}=\pin{f}{\xi_n'}_+,
\end{equation}
for all $f\in \D$. Hence, by Proposition \ref{pro_lower_frame}, $\xi'$ is a frame of $\dom(\xi)$.\\

\no Now, assume that $\phi$ is a frame of $\dom(\xi)$. 
A natural question arises: does there exist a lower semi-frame $\xi$ of $\H$  such that $\phi$ is the frame constructed from $\xi$ in the described way?
To answer this question, we note that $\pint_+$ is a positive closed sesquilinear form on $\D$. By Kato's  representation theorems there exists a positive self-adjoint operator $R$  such that $\dom(R)\subseteq \dom(\xi)$, $\dom(R^\mez)=\dom(\xi)$, $0\in \rho(R)$ and 
$$
\pin{f}{g}_+=\pin{R^\mez f}{R^\mez g}, \qquad \forall f,g \in \dom(\xi),
$$
\begin{equation}
\label{def_R}
\pin{f}{g}_+=\pin{f}{Rg}, \qquad \forall f \in \dom(\xi),g\in \dom(R).
\end{equation}
By (\ref{Riesz_qrp}) we have $\xi_n=R\xi_n'$ for all $n\in \N$. Then we can state the following.

\begin{pro}
	Let $\phi$ be a frame of $\dom(\xi)[\pint_+]$.
	\begin{enumerate}[label=\emph{(\roman*)}]
		\item There exists a lower semi-frame $\xi=\{\xi_n\}$ on $\dom(\xi)$ such that $\phi=\xi'$ if, and only if, $\phi_n\in \dom(R)$ for all $n\in \N$.
		\item If $\phi=\xi'$ for some lower semi-frame  $\xi=\{\xi_n\}$ on $\dom(\xi)$, then $\xi_n=R\phi_n$ for all $n\in\N$.
	\end{enumerate}
\end{pro}

As an application, we show two particular ways to construct lower semi-frames.

\begin{exm}
	\begin{enumerate}[label={(\roman*)}]
	\item
	Let $S$ be a closed operator on $\H$ with dense domain $\D$, and let $\pin{f}{g}_S=\pin{f}{g}+\pin{Sf}{Sg}$, $f,g\in \D$. Then, $\D:=\D[\pint_S]$ is a Hilbert space continuously embedded in $\H$. The operator associated to $\pint_S$ is $I+S^*S$; hence if $\{e_n\}$ is an ONB of $\D$, contained in $D(|S|^2)$, then $\{(I+S^*S)e_n\}$ is a lower semi-frame of $\H$ on $\D$.
	\item
	A slight different argument leads to another example. Assume also that $0\in \rho(S)$ then $\{f|g\}_S=\pin{Sf}{Sg}$, $f,g\in \D$, is a inner product inducing the same topology of $\D$. The associated operator to $\{\cdot|\cdot\}_S$ is $S^*S$, therefore if $\{e_n\}$ is an ONB as above, then $\{(S^*S)e_n\}$ is a lower semi-frame of $\H$ on $\D$.
\end{enumerate} 
\end{exm}


\section{Sesquilinear forms associated to two sequences}
\label{sec:2seq}

In this section we consider two sequences $\xi=\{\xi_n\},\eta=\{\eta_n\}$ of  $\H$. In addition to the analysis and synthesis operators of both sequences one can also define the operator $S_{\xi,\eta}$ on
$$\dom(S_{\xi,\eta}):=\left \{f\in \H: \sum_{n=1}^\infty \pin{f}{\xi_n}\eta_n \text{ exists in }\H \right \}$$
as $S_{\xi,\eta} f =\sum_{n=1}^\infty \pin{f}{\xi_n}\eta_n$, for $f\in \dom(S_{\xi,\eta})$. This operator is actually a {\it multiplier} in the sense of \cite{Balazs_mult}.\\
Clearly, $D_\eta C_\xi \subseteq S_{\xi,\eta}$. However, unlike the case when $\xi=\eta$, the following example demonstrates that the equality $D_\eta C_\xi = S_{\xi,\eta}$ does not always hold.

\begin{exm}
	\label{exm_DC<S}
	Let $\H$ be $\{e_n\}$ an ONB on $\H$.
	Let 
	$$
	\xi:=\left \{e_1,e_1,e_2,2e_2,\dots,e_n,n e_n, \dots\right \}
	$$
	and
	$$
	\eta:=\left \{e_1,0,e_2,0,\dots,e_n,0, \dots\right \}.
	$$
	In particular, $\xi$ is a lower semi-frame and $\eta$ is a frame. 	The operator $D_\eta C_\xi$ is defined on $\dom(D_\eta C_\xi)=\dom(C_\xi)\neq \H$ and acts as $D_\eta C_\xi f =\sum_{n=1}^\infty f_{n}e_n=f $ for $f\in \dom(D_\eta C_\xi)$. However, $S_{\xi,\eta}$ is equal to the identity operator $I$ on $\H$.
\end{exm}

\no  
If $S_{\xi,\eta}$ is invertible, then we have the following reconstruction formula in strong sense
$$
f=S_{\xi,\eta} S_{\xi,\eta}^{-1}f=\sum_{n=1}^\infty \pin{S_{\xi,\eta}^{-1}f}{\xi_n}\eta_n, \qquad \forall f\in R(S_{\xi,\eta}).
$$
Note also that when $\xi,\eta$ are Bessel sequences then $\dom(S_{\xi,\eta})=\H$ and $S_{\xi,\eta}$ is bounded. 
In addition if $\xi,\eta$ are Bessel sequences and $S_{\xi,\eta}=I$, then $\xi,\eta$ are in particular frames (see \cite[Proposition 6.1]{Casazza_dual}).\\

\no Now, we turn our attention to the sesquilinear form defined by two sequences, i.e., 
$$
\O_{\xi,\eta}(f,g)=\sum_{n=1}^\infty \pin{f}{\xi_n}\pin{\eta_n}{g},
$$
which need not be nonnegative nor symmetric. The series above is not necessarily unconditionally convergent. In analogy to the case of one sequence, we can consider also the sesquilinear form 
$$
\Theta_{\xi,\eta}(\{c_n\},\{d_n\})=\sum_{i,j\in \N} c_i\ol{d_j}\pin{\xi_i}{\eta_j}.
$$
However, we will only focus on $\O_{\xi,\eta}$ in this paper.
Our task is to consider this form on some domain $\D_1\times \D_2$ such that $\D_2$ is dense and make it a $0$-closed form. In this way, by Theorem \ref{th_rapp_risol}, the operator $\mathcal{T}$ associated to $\O_{\xi,\eta}$ on $\D_1\times \D_2$ is 
closed, invertible with bounded inverse and with domain $\dom(T)\subseteq \D_1$.
This leads to a reconstruction formula in weak sense, i.e., if $h\in \H$  then for all $g\in\D_2$ 
\begin{align}
\pin{h}{g}&=\pin{\mathcal{T}\mathcal{T}^{-1}h}{g}=\O_{\xi,\eta}(\mathcal{T}^{-1}h,g) \nonumber\\
&=\sum_{n=1}^\infty \pin{\mathcal{T}^{-1}h}{\xi_n}\pin{\eta_n}{g}=\sum_{n=1}^\infty \pin{h}{(\mathcal{T}^{-1})^*\xi_n} \pin{\eta_n}{g} 
\label{rec_weak_1}
\end{align}
or, equivalently, $ g=(w)\sum_{n=1}^\infty \pin{g}{\eta_n}(\mathcal{T}^{-1})^*\xi_n$.
If also $\D_1$ is dense, then $\dom(\mathcal{T})$ is dense, $\mathcal{T}^*$ is the operator associated to $\O_{\xi,\eta}^*$ (i.e., $\O_{\eta,\xi}$ on $\D_2\times \D_1$) and $(\mathcal{T}^{-1})^*=(\mathcal{T}^*)^{-1}$. Therefore, in a similar way, for $f\in \D_1$ and $h\in \H$
\begin{equation*}
\label{rec_weak_2}
\pin{h}{f}= \sum_{n=1}^\infty \pin{h}{\mathcal{T}^{-1}\eta_n}\pin{\xi_n}{f},
\end{equation*}
i.e., $f=(w)\sum_{n=1}^\infty \pin{f}{\xi_n}\mathcal{T}^{-1}\eta_n$.
\no The idea of looking at $0$-closed forms is also justified by the following consideration. Assume that $(\xi,\eta)=(\{\xi_n\},\{\eta_n\})$ is a reproducing pair on $\H$ in the sense of \cite{AST,Anto_Tp,Speck_Bal_15,Speck_Bal_16}. It is easy to see that $\O_{\xi,\eta}$ is $0$-closed. However, it is well-defined and bounded on $\H$. Our approach then gives a generalization of the notion of reproducing pairs in cases where $\O_{\xi,\eta}$ is unbounded.


\subsection{Bounded case}

First of all, let us study the case when $\O_{\xi,\eta}$ is bounded. As proved below, this is always the case when the form is defined on the whole space, in analogy to the situation with one sequence.

\begin{pro}
	Let $\xi,\eta$ be sequences such that $\O_{\xi,\eta}$ is defined on $\H\times \H$. Then $\O_{\xi,\eta}$ is bounded.
\end{pro}
\begin{proof}
	Denote by $\O_{\xi,\eta}^k(f,g)= \sum_{n=1}^k \pin{f}{\xi_n}\pin{\eta_n}{g}$, for all $f,g\in \H$. Clearly, there exists $T_k\in \B(\H)$ such that $\O_{\xi,\eta}^k(f,g)=\pin{T_k f}{g}$. 
	The Banach–Steinhaus theorem ensures that the operator $T_{\xi,\eta}$ associated to $\O_{\xi,\eta}$ has domain the whole $\H$. Applying the same theorem to the linear functional $f\mapsto \pin{f}{T_k^* g}=\O_{\xi,\eta}^k(f,g)$ for $g\in\H$, we get
	$\displaystyle \pin{T_{\xi,\eta} f}{g}=\lim_{k\to \infty} \pin{T_kf}{g}=\lim_{k\to \infty} \pin{f}{T_k^* g}=\pin{f}{h}$ for some $h\in \H$. Therefore $D(T_{\xi,\eta}^*)=\H$, i.e., $T_{\xi,\eta}$ (and consequently $\O_{\xi,\eta}$) is bounded.	
\end{proof}

\no As a consequence we can say that $(\xi,\eta)$ is a reproducing pair if and only if $\O_{\xi,\eta}$ is defined on $\H\times \H$ and it is $0$-closed.\\
We recall again that the operator associated to a (bounded) form $\O_\xi$ on $\H\times \H$ is the operator $S_\xi$. Nevertheless, as shown in the next example, when we turn to two sequences $\xi,\eta$ such that $\O_{\xi,\eta}$ is defined on $\H\times \H$ (and therefore bounded), then the operator $T_{\xi,\eta}$ associated to $\O_{\xi,\eta}$ (which is an element of $\B(\H)$) need not be $S_{\xi,\eta}$. More precisely, it is defined as $T_{\xi,\eta} f =(w)\sum_{n=1}^\infty \pin{f}{\xi_n}\eta_n$ for $f\in \H$.

\begin{exm}
	Let again $\{e_n\}$ be an ONB of $\H$. We set
	\begin{align*}
	\xi&=\{e_1,e_1,-e_1,e_2,e_1,-e_1,e_3,e_1,-e_1,\dots\},\\
	\eta&=\{e_1,e_1,e_1,e_2,e_2,e_2,e_3,e_3,e_3,\dots\}.
	\end{align*}
	It is easy to see that $\O_{\xi,\eta}$ is well defined on $\H \times \H$ and $\O_{\xi,\eta}(f,g)=\pin{f}{g}$ for $f,g\in \H$. Thus the operator associated to $\O_{\xi,\eta}$ is the identity operator. However, $\dom(S_{\xi,\eta})=\{f\in \H: \pin{f}{e_1}=0\}$.  
	Note also that $\dom(S_{\eta,\xi})=\H$ and $S_{\eta,\xi}=I$. Then it is the operator associated to $\O_{\eta,\xi}$ (which is exactly $\O_{\xi,\eta}$ since it is symmetric). 
\end{exm}

\no It is worth to mention that if $\xi$ is a Bessel sequence with upper bound $B$ and $\eta$ a sequence such that $\O_{\xi,\eta}=\iota$, then $\eta$ is a lower semi-frame with lower bound $B^{-1}$ (the proof is analogous to the one of Lemma 2.5 of \cite{Anto_Bal} in the discrete version).

\subsection{General case}

Now we return to consider two sequences $\xi,\eta$ generating a generic form $\O_{\xi,\eta}$. This form is clearly defined on $\dom(\xi)\times\dom(\eta)$ and 
$$
\O_{\xi,\eta}(f,g)=\pin{C_\xi f}{C_\eta g}_2, \qquad \forall f\in \dom(\xi),g\in \dom(\eta).
$$
With this domain $\O_{\xi,\eta}$ is a q-closed sesquilinear form. Indeed the definition is satisfied considering the graph norms
$\n{f}_{C_\xi}= (\n{C_\xi f}^2+\n{f}^2)^\mez$ and $\n{g}_{C_\eta}= (\n{C_\eta g}^2+\n{g}^2)^\mez$ on $\dom(\xi)$ and $\dom(\eta)$, respectively. 

Moreover, $N(\O_{\xi,\eta})=\{f\in \dom(\xi): C_\xi f \in R(C_\eta)^\perp\}$. Other properties (in particular equivalent conditions for $\O_{\xi,\eta}$ to be $0$-closed) of this form are stated in the next theorem. Note that if $V,W$ are closed subspaces of a Hilbert space then we denote by $V\dotplus W$ the direct sum of $V$ and $W$.

\begin{theo}
	\label{th_cns_0clos}
	Let us consider $\O_{\xi,\eta}$ on the domain $\dom(\xi)\times\dom(\eta)$. 
	The following statements are equivalent.
			\begin{enumerate}[label=\emph{(\roman*)},ref=(\roman*)]
			\item[\emph{(a)}] $\O_{\xi,\eta}$ is $0$-closed.  
			\item[\emph{(b)}] $\xi,\eta$ are lower semi-frames and 
			$R(C_\xi) \dotplus R(C_\eta)^\perp= l_2$.
			\item[\emph{(c)}] $\xi,\eta$ are lower semi-frames and 
			$R(C_\eta) \dotplus R(C_\xi)^\perp= l_2$.	
		\end{enumerate}
	Assume that $\dom(\eta)$ is dense. Then the operator associated to $\O_{\xi,\eta}$ is $C_\eta^*C_\xi$.
\end{theo}
\begin{proof}
	Firstly, note that if $\xi,\eta$ are lower semi-frames, then $\nor_{C_\xi}$ and $\nor_{C_\eta}$ are equivalent to the norms given by $\n{f}_\xi=\n{C_\xi f}$ and $\n{g}_\eta=\n{C_\eta g}$ with $f\in \dom(\xi),g\in \dom(\eta)$, respectively.\\
	Assume that $0$-closed. By Lemma \ref{lem_cns_0clos}(iv) 
	\begin{align*}
	c_1\n{f}\leq c_1\n{f}_{C_\xi}\leq \sup_{\n{g}_{C_\eta}=1} |\O_{\xi,\eta}(f,g)|\leq \n{C_\xi f}, \qquad \forall f \in \dom(\xi).
	\end{align*}
	This means that $\xi$ (and in the same way $\eta$) is a lower semi-frame (see \cite[Proposition 4.1]{Classif}). 
	Taking into account the equivalence of norms above, we can rewrite the inequality in Lemma \ref{lem_cns_0clos}(iv) as follows 
	\begin{align}
	\label{dis_1}
	c_1\n{f}_{\xi}\leq \sup_{\n{g}_{\eta}=1} |\O_{\xi,\eta}(f,g)|, \qquad \forall f \in \dom(\xi),\\
	\label{dis_2}
	c_2\n{g}_{\eta}\leq \sup_{\n{f}_{\xi}=1} |\O_{\xi,\eta}(f,g)|, \qquad \forall g \in \dom(\eta).
	\end{align}
	where $c_1,c_2>0$.
	Moreover, denoting by $P_{R(C_\xi)}$ and $P_{R(C_\eta)}$ the orthogonal projections onto the closed ranges $R(C_\xi)$ and $R(C_\eta)$, respectively, one has
	\begin{align}
	\sup_{\n{g}_{\eta}=1} |\O_{\xi,\eta}(f,g)| &= \sup_{\n{g}_{\eta}=1} |\pin{C_\xi f}{C_\eta g}| \nonumber\\
	&=\sup_{\n{C_\eta g}=1} |\pin{P_{R(C_\eta)} C_\xi f}{C_\eta g}| \nonumber\\
	&=\n{P_{R(C_\eta)}C_\xi f}. \label{sup_angle}
	\end{align}
	Then (\ref{dis_1}) and (\ref{dis_2}) carry to $\displaystyle \inf_{\n{C_\xi f}=1} \n{P_{R(C_\eta)}C_\xi f}>0$. In a similar way $\displaystyle \inf_{\n{C_\eta g}=1} \n{P_{R(C_\xi)}C_\eta g}>0$ holds. Theorem 2.3 of \cite{Tang} implies that $R(C_\xi) \dotplus R(C_\eta)^\perp= l_2$ or equivalently that $R(C_\eta) \dotplus R(C_\xi)^\perp= l_2$.
	
	Conversely, let and $\xi,\eta$ be lower semi-frames such that $R(\xi) \dotplus R(C_\eta)^\perp= l_2$ or $R(\eta) \dotplus R(C_\xi)^\perp= l_2$. Again by Theorem 2.3 of \cite{Tang},  we have that
	$$ \inf_{\n{C_\xi f}=1} \n{P_{R(C_\eta)}C_\xi f}>0 \quad \text{and} \quad \inf_{\n{C_\eta g}=1} \n{P_{R(C_\xi)}C_\eta g}>0.$$
	Therefore, equality $(\ref{sup_angle})$ implies that (\ref{dis_1}) and (\ref{dis_2}) hold and $\O_{\xi,\eta}$ is $0$-closed.\\
	Finally, the statement about the associated operator is easy to prove.
\end{proof} 

\begin{rem}
	\begin{enumerate}
		\item If $D_\eta$ is closed, then the operator associated to a form  $\O_{\xi,\eta}$ on $\dom(\xi)\times \dom(\eta)$ with $\dom(\eta)$ dense coincides with $D_\eta C_\xi$.
		\item If $\dom(\eta)$ (resp., $\dom(\xi)$) is dense, then $R(C_\xi) \dotplus R(C_\eta)^\perp= l_2$ (resp., $R(C_\eta) \dotplus R(C_\xi)^\perp= l_2$) is equivalent to $R(C_\xi) \dotplus N(C_\eta^*)= l_2$ (resp., $R(C_\eta) \dotplus N(C_\xi^*)= l_2$).  
	\end{enumerate}
\end{rem}

\no By Theorem \ref{th_rapp_risol} we get the next characterization.

\begin{cor}
	\label{cor_0_in_rho}
	Assume that $\dom(\eta)$ is dense. Then $0\in \rho(C_\eta^*C_\xi)$ if the following equivalent statements are satisfied
	\begin{enumerate}
		\item[\emph{(a)}] $\xi,\eta$ are lower semi-frames and $R(C_\xi) \dotplus N(C_\eta^*)= l_2$. 
		\item[\emph{(b)}] $\xi,\eta$ are lower semi-frames and $R(C_\eta) \dotplus R(C_\xi)^\perp= l_2$.	
	\end{enumerate}
\end{cor}

\no If $\O_{\xi,\eta}$ is solvable (in particular, $\lambda$-closed) and $\dom(\eta)$ is dense, then by Theorem \ref{th_rapp_risol} $C_\eta^*C_\xi$ is closed and, moreover, densely defined if $\D(\xi)$ is dense. Otherwise, $C_\eta^*C_\xi$ need not be densely defined nor closed. Note also that if $\xi=\eta$ we regain that $\O_{\xi}$ is $0$-closed if and only if $\xi$ is a lower semi-frame (see Proposition \ref{car_form_1_seq}). A particular case of Theorem \ref{th_cns_0clos} occurs when the domains (or one of them) coincide with the whole space.

\begin{cor} 
	Let $\xi,\eta$ be two sequences of $\H$. The following statements hold.
	\begin{enumerate}[label=\emph{(\roman*)}]
		\item If $\dom(\xi)=\H$ and $\O_{\xi,\eta}$ is $0$-closed, then $\xi$ is a frame of $\H$.
		\item If $\dom(\xi)=\dom(\eta)=\H$, then $(\xi,\eta)$ is a reproducing pair if and only if $\xi,\eta$ are frames and $R(C_\xi) \dotplus N(D_\eta)= l_2$ (resp., $R(C_\eta) \dotplus N(D_\xi)= l_2$).
	\end{enumerate} 
	
\end{cor}


\no Under condition (b) (or (c)) of Theorem \ref{th_cns_0clos} and that $\dom(\eta)$ is dense, formula (\ref{rec_weak_1}) holds with $\D_1=\dom(\xi)$, $\D_2=\dom(\eta)$ and $\mathcal{T}=C_\eta^* C_\xi$. Moreover, (\ref{rec_weak_1}) can be improved as follows. 
\begin{cor}
	Let $\xi,\eta$ be two sequences of $\H$ with  $\dom(\eta)$ is dense and $\mathcal{T}=C_\eta^*C_\xi$. If conditions \emph{(b)} or \emph{(c)} of Theorem \ref{th_cns_0clos} are satisfied, then  $\{(\mathcal{T}^{-1})^* \xi_n\}$ is a Bessel sequence of $\H$ and 
	\begin{equation*}
	g=\sum_{n=1}^\infty \pin{g}{\eta_n}(\mathcal{T}^{-1})^*\xi_n, \qquad \forall g\in \dom(\eta).
	\end{equation*}
\end{cor}
\begin{proof}
	By Corollary \ref{cor_0_in_rho}, $0\in \rho(\mathcal{T})$; hence $C_\xi$ is injective with closed range $R(C_\xi)$ and $R(C_\eta^*)=\H$. Moreover, $\mathcal{T}=GC_\xi$ where $G$ is the restriction of $C_\eta^*$ on $\dom(C_\eta^*)\cap R(C_\xi)$. Note that $G$ is closed, invertible operator and $R(G)=\H$, i.e. $G^{-1}\in \B(\H)$. Thus, for $f\in \H$ we have $\mathcal{T}^{-1} f\in \dom(\xi)$ and
	\begin{align*}
	\sum_{n=1}^\infty |\pin{f}{(\mathcal{T}^{-1})^*\xi_n}|^2=\|C_\xi \mathcal{T}^{-1}f\|^2_2=\|G^{-1}f\|^2\leq \|G^{-1}\|^2\|f\|^2.
	\end{align*}
	Hence $\{(\mathcal{T}^{-1})^* \xi_n\}$ is a Bessel sequence. 
	Thus, for $g\in \dom(\eta)$, $\sum_{n=1}^\infty \pin{g}{\eta_n}(\mathcal{T}^{-1})^*\xi_n$ is convergent and, in particular, by \eqref{rec_weak_1} it is convergent to $g$.
\end{proof}

\subsection{Maximality of domains}

The domain $\dom(\xi) \times \dom(\eta)$ is not always a maximal domain on which $\O_{\xi,\eta}$ can be defined. For instance, let us consider
	\begin{equation}
	\label{exm_extens}
	\xi=\{n e_n\} \text{ and } \eta=\{n^{-1}e_n\}, \text{ where  $\{e_n\}$ is an ONB.} 
	\end{equation}
	Clearly, $\O_{\xi,\eta}$ can be defined on $\H\times \H$, which is larger than $\dom(\xi) \times \dom(\eta)$.\\
	We would stress that there exist more significant examples than the previous one. In \cite{Speck_Bal_16} for a Bessel Gabor sequence $\xi$ the authors found a non Bessel sequence $\eta$ such that $\O_{\xi,\eta}$ is defined on $L^2(\R) \times L^2(\R)$, while  $\dom(\xi) \times \dom(\eta)$ is a proper subspace of  $L^2(\R) \times L^2(\R)$.
	
	This is the general situation for two sequences which are one the dual of the other one. Recall that a sequence $\eta$ is a {\it dual} of a sequence $\xi$ if 
	$$
	f=\sum_{n=1}^\infty \pin{f}{\xi_n}\eta_n=\sum_{n=1}^\infty \pin{f}{\eta_n}\xi_n, \qquad \forall f\in \H,
	$$
	i.e., $S_{\xi,\eta}=S_{\eta,\xi}=I$. The sesquilinear form $\O_{\xi,\eta}$ is clearly defined on $\H\times \H$.
	
	Other possible choices of the domain of $\O_{\xi,\eta}$ for two sequences $\xi,\eta$ are given easily as follows. Let $\{\alpha_n\}$ be a sequence of nonzero complex numbers and let $\xi'=\{\alpha_n \xi_n\}$ and $\eta'=\{\ol{\alpha_n}^{-1}\eta_n\}$. Then 
	$\O_{\xi,\eta}$ is defined on $\dom(\xi')\times \dom(\eta')$. Therefore, as in the case of example (\ref{exm_extens}), with an opportune sequence $\{\alpha_n\}$ the form $\O_{\xi,\eta}$ may be defined on a domain larger than $\dom(\xi) \times \dom(\eta)$.
	

The next result will be useful in the sequel.

\begin{lem}
	\label{lem_W^perp}
	Let $W$ be a closed subspace of $l_2$ with $\dim W^\perp <\infty$ and  $\{a_n\}$ be a complex sequence such that 
	$$
	\sum_{n=1}^\infty a_n b_n \text{ is convergent for all } \{b_n\}\in W.
	$$
	Then $\{a_n\}\in l_2$.   
\end{lem}
\begin{proof}
	As it is well-known \cite[Example 34.2]{Heuser}, the statement is true when $W=l_2$. Assume now that $W$ is a proper subspace of $l_2$. Let $m=\dim W^\perp$ and $\{d^1,\dots,d^m\}$ be a basis of $W^\perp$. We will prove that, for $k=\{k_n\}\in l_2$, 
	$\sum_{n=1}^\infty a_n k_n$ is convergent. \\
	Let $p\geq m$ be an index such that the vectors $(d^i_1,\dots, d^i_p)$, for $i=1,\dots, m$, are independent. Such a integer exists because $\{d^1,\dots,d^m\}$ are independent. 
	Put $h=\{h_1,\dots, h_p, k_{p+1},k_{p+2},\dots\}$ where $h_1,\dots, h_p$ are complex numbers. This is an element of $l_2$. Our purpose is to find complex numbers $h_1,\dots, h_j$ such that, in particular, $h\in W$. Since $W$ is closed, it is enough to impose that $h\perp d^i$ for $i=1,\dots,m$. The conditions
	$$
	0=\pin{h}{d^i}=\sum_{n=1}^p h_n\ol{d^i_n} + \sum_{n=p+1}^\infty k_n\ol{d^i_n}, \qquad \forall i=1,\dots, m,
	$$
	constitute a linear system in the variables $h_1,\dots, h_p$. This system admits solutions since the $m$ vectors $(d^i_1,\dots, d^i_p)$ are independent. \\
	For such a solution, $h\in W$ and therefore $\sum_{n=1}^\infty a_n h_n$ is convergent by hypothesis. But 
	$\sum_{n=1}^\infty a_n k_n=\sum_{n=1}^p a_n (k_n-h_n)+ \sum_{n=1}^\infty a_n h_n$. In conclusion  $\sum_{n=1}^\infty a_n k_n$ is convergent for all $k=\{k_n\}\in l_2$ and this implies that $\{a_n\}\in l_2$.  
\end{proof}

\no Now we discuss about the maximality of the domain of $\O_{\xi,\eta}$ where $\xi=\{\xi_n\}$ and $\eta=\{\eta_n\}$ are two sequences. Let $Y$ be a subspace of $\H$. Set
$$
\mathcal{X}(Y):=\left \{f\in \H: \sum_{n=1}^\infty \pin{f}{\xi_n}\pin{\eta_n}{g} \text{ exists in $\H$  for all } g\in Y \right \}.
$$
An analog definition can be given by symmetry for a fixed first component. Some properties of $\mathcal{X}$ are listed in the next proposition, whose easy proof is omitted.

\begin{pro}
	The following statements hold.
	\begin{enumerate}[label=\emph{(\roman*)}]
		\item $\mathcal{X}(Y)$ is the greatest subspace $\D_1$ of $\H$ for which $\O_{\xi,\eta}$ can be defined on $\D_1 \times Y$.
		\item The map $\mathcal{X}$ is decreasing, i.e., if $Y_1\subseteq Y_2$ are two subspaces of $\H$, then $\mathcal{X}(Y_1)\supseteq \mathcal{X}(Y_2)$.
		\item $\mathcal{X}(\H)=\{f\in \H:(w)\sum_{n=1}^\infty \pin{f}{\xi_n}\eta_n \text{ exists in }\H \}$.
		\item $\dom(\xi)\subseteq \mathcal{X}(\dom(\eta))$. If $\dim R(C_\eta)^\perp <\infty$ and $R(C_\eta)$ is closed, then $\mathcal{X}(\dom(\eta))=\dom(\xi)$.
		\item $\mathcal{X}(\dom(S_{\eta,\xi}))=\mathcal{X}(\{g\in \H:(w)\sum_{n=1}^\infty \pin{g}{\eta_n}\xi_n \text{ exists in }\H  \})=\H$.
	\end{enumerate} 	
\end{pro}

\no Note that point (iv) is a consequence of Lemma \ref{lem_W^perp} and occurs, for instance, if $\eta$ is a Riesz-Fischer sequence. A consequence of this proposition is: if $\dim N(C_\xi^*),\dim N(C_\eta^*)<\infty$ and $R(C_\xi),R(C_\eta)$ are closed, then $\dom(\eta)\subseteq \D_2$ and $\dom(\xi)\subseteq \D_1$ imply $\D_1=\dom(\xi)$ and $\D_2=\dom(\eta)$. That is, $\dom(\xi)\times \dom(\eta)$ is the greatest domain in this case.

Suppose now that $Y$ is dense in $\H$. Denote by $\mathcal{T}(Y)$ the operator associated to $\O_{\xi,\eta}$ on $\mathcal{X}(Y)\times Y$. In this way we can define an operator-valued map $\mathcal{T}$ defined on the family of dense subspaces of $\H$.

\begin{pro}
	The following statements hold.
	\begin{enumerate}[label=\emph{(\roman*)}]
		\item The map $\mathcal{T}$ is decreasing, i.e., if $Y_1\subseteq Y_2$ are two dense subspaces of $\H$, then $\mathcal{T}(Y_1)\supseteq \mathcal{T}(Y_2)$.
		\item $\dom(\mathcal{T}(\H))=\mathcal{X}(\H)=\{f\in \H:(w)\sum_{n=1}^\infty \pin{f}{\xi_n}\eta_n \text{ exists in }\H \}$ and $\mathcal{T}(\H)f=(w)\sum_{n=1}^\infty \pin{f}{\xi_n}\eta_n$ for all $f\in \dom(\mathcal{T}(\H))$.
		\item If $\dom(S_{\eta,\xi})$ is dense, then  $\mathcal{T}(\dom(S_{\eta,\xi}))=S_{\eta,\xi}^*$.
	\end{enumerate} 
\end{pro}

\section{Examples}
\label{sec:exm}

\subsection{Weighted Riesz basis and canonical dual}
Let $\phi:=\{\phi_n\}$ be a Riesz basis of $\H$. Then there exists a bounded bijective operator $V\in \B(\H)$ such that $\phi_n =V e_n$ for all $n\in \N$ (see  \cite{Ole}). The analysis operator $C_\phi$ is defined for $f\in \H$ as $C_\phi f =\{\pin{V^*f}{e_n}\}$.  
The {\it canonical dual} of $\phi$ is the sequence $\psi:=\{(V^{-1})^* e_n\}$. Therefore $\phi$ and $\psi$ are biorthogonal.\\

\no Now let $\alpha=\{\alpha_n\}$ be a complex sequence and $\O^\alpha_{\phi,\psi}$ the sesquilinear form
\begin{equation}
\label{form_weig_Riesz}
\O^\alpha_{\phi,\psi} (f,g)=\sum_{n=1}^\infty \alpha_n \pin{f}{\phi_n}\pin{\psi_n}{g}.
\end{equation}
We can define this form on domains of the type $\dom(\xi)\times \dom(\eta)$ where $\xi=\{\beta_n \phi_n\}$, $\eta=\{\gamma_n \psi_n\}$ and $\ol{\beta_n}\gamma_n=\alpha_n$ for all $n\in \N$. In other words, we can consider $\O^\alpha_{\phi,\psi}$ as a form $\O_{\xi,\eta}$ where $\xi,\eta$ are sequences as above. \\ 

\no First of all, let us determine the (densely defined) analysis operator $C_\xi$ of $\xi=\{\beta_n \phi_n\}$ and its adjoint $C_\xi^*$. The operator $C_\xi$ is defined on the dense domain
$
\dom(\xi)=\left \{g\in \H: \sum_{n=1}^\infty |\beta_n|^2|\pin{V^*g}{e_n}|^2 <\infty \right \}
$
as $C_\xi g =\{\ol{\beta_n}\pin{V^*g}{e_n}\}$. 
The adjoint $C_\eta^*$ has domain $\dom(C_\eta^*)=l_2\cap l_2(\beta^2)$. Indeed, let $\{c_n\}\in l_2$. The linear functional 
\begin{equation}
\label{fun_Riesz}
g \mapsto \sum_{n=1}^\infty c_n\pin{g}{\xi_n}= \sum_{n=1}^\infty \ol{\beta_n} c_n\pin{V^*g}{ e_n}
\end{equation}
is clearly bounded on $\dom(C_\xi)$ if $\{c_n\}\in l_2(\beta^2)$. Conversely assume that (\ref{fun_Riesz}) is bounded on $\dom(C_\xi)$ with bound $M>0$. Since $V^*$ is a bijection of $\H$, for all $k\in \N$ there exists $h\in \H$ such that 
$ V^* h =\sum_{n=1}^k \beta_n\ol{c_n} e_n$. Therefore for all $k\in \N$
\begin{align*}
\sum_{n=1}^k |\beta_n c_n|^2&=\left |\sum_{n=1}^\infty c_n\pin{h}{\xi_n}\right | \leq M \n{h}\\
&\leq M \n{{V^*}^{-1}}\n{V^*h} \\
&=M\n{{V^*}^{-1}} \left (\sum_{n=1}^k |\beta_n c_n|^2\right )^\mez.
\end{align*}
This means that $\{c_n\}\in l_2(\beta^2)$. Moreover, it is easy to see that $D_\xi=C_\xi^*$. If $\eta=\{\gamma_n \psi_n\}$, then in the same way $\dom(C_\eta^*)=\dom(D_\eta)=l_2\cap l_2(\gamma^2)$. \\

\no Coming back to the study of (\ref{form_weig_Riesz}), put $\xi=\{\beta_n \phi_n\}$ and $\eta=\{\gamma_n \psi_n\}$ and $\ol{\beta_n}\gamma_n=\alpha_n$ for all $n\in \N$. The operator associated to $\O_{\xi,\eta}$ is $C_\eta^*C_\xi=D_\eta C_\xi \subseteq S_{\xi,\eta}$, which is defined by 
$$
D_\eta C_\xi f=\sum_{n=1}^\infty \alpha_n \pin{f}{\phi_n}\psi_n
$$
on the domain 
\begin{align*}
\dom(D_\eta C_\xi)&=\{f\in \dom(C_\xi): C_\xi f \in \dom(D_\eta)\}\\
&=\left \{f\in \H:\sum_{n=1}^\infty (|\alpha_n|^2+|\beta_n|^2)|\pin{f}{\phi_n}|^2<\infty \right \}.
\end{align*}
In general, $D_\eta C_\xi$ is densely defined and closable but it need not be closed. In the following we adopt the choice $\xi=\{\phi_n\}$ and $\eta=\{\alpha_n \psi_n\}$.



\begin{pro}
	Let us consider $\O_{\xi,\eta}$ on the domain $\dom(\xi)\times\dom(\eta)$ where $\xi=\{\phi_n\}$ and $\eta=\{\alpha_n \psi_n\}$. The following statements hold.
	\begin{enumerate}[label=\emph{(\roman*)}]
		\item The operator associated to $\O_{\xi,\eta}$ is $S_{\xi,\eta}$. Moreover, $S_{\xi,\eta}$ is defined by
		$$
		S_{\xi,\eta} f =\sum_{n=1}^\infty \alpha_n \pin{f}{\phi_n}\psi_n
		$$
		on the domain
		$$\dom(S_{\xi,\eta}):=\left \{f\in \H: \sum_{n=1}^\infty |\alpha_n|^2|\pin{f}{\phi_n}|^2<\infty \right \}.$$
		\item $\O_{\xi,\eta}$ is $0$-closed if and only if $\inf_n |\alpha_n|>0$.
		\item $\O_{\xi,\eta}$ is solvable. In particular, $\O_{\xi,\eta}$ is $\lambda$-closed if and only if $\lambda \notin \ol{\{\alpha_n\}}$, the closure of $\{\alpha_n\}$. 
	\end{enumerate}
\end{pro}
\begin{proof}
	\begin{enumerate}[label={(\roman*)}]
		\item
		Since $\dom(C_\xi)=\H$ we have the equality $D_\eta C_\xi = S_{\xi,\eta}$. Moreover 
		\begin{align*}
		\dom(S_{\xi,\eta})=\dom(D_\eta C_\xi)=\left \{f\in \H: \sum_{n=1}^\infty |\alpha_n|^2|\pin{f}{\phi_n}|^2<\infty \right \}.
		\end{align*}	
		\item By Theorem \ref{th_cns_0clos} if $\O_{\xi,\eta}$ is $0$-closed, then $\eta$ is a lower semi-frame, i.e., $\inf_n |\alpha_n|>0$. Conversely, the condition $\inf_n |\alpha_n|>0$ ensures that $\eta$ is a lower semi-frame and $N(D_\eta)=\{0\}$. Taking into account that $C_\xi$ is bijective by \cite[Proposition 4.1]{Classif}, one has $R(C_\xi)\dotplus N(D_\eta)=l_2$ and $\O_{\xi,\eta}$ is $0$-closed by Theorem \ref{th_cns_0clos}.
		\item  
		Let 
		$$\xi'=\{\xi_1,\sigma_1\xi_1,\dots,\xi_n,\sigma_n\xi_n,\dots\} \text{ and } \eta'=\{\eta_1,\psi_1,\dots,\eta_n,\psi_n,\dots\},$$
		where $\sigma_n = -\alpha_n+1$ if $|\alpha_n|\leq 1$ and $\sigma_n=0$ if $|\alpha_n|> 1$. Therefore $\O_{\xi',\eta'}(f,g)=\sum_{n=1}^\infty (\alpha_n+\sigma_n) \pin{f}{\phi_n}\pin{\psi_n}{g}$ and $|\alpha_n+\sigma_n|\geq 1$ for all $n\in \N$. By the point (ii), $\O_{\xi',\eta'}$ is $0$-closed. Since $\O_{\xi',\eta'}=\O_{\xi,\eta}+\Upsilon$, where $\Upsilon$ is the bounded form $\Upsilon(f,g)=\sum_{n=1}^\infty \sigma_n \pin{f}{\phi_n}\pin{\psi_n}{g}$, $\O_{\xi,\eta}$ is solvable. In particular, if $\lambda \in \C$, taking $\sigma_n=-\lambda$ for all $n\in \N$ we recover that $\O_{\xi,\eta}$ is $\lambda$-closed if and only if  $\lambda \notin \ol{\{\alpha_n\}}$.		\qedhere
	\end{enumerate}
\end{proof}

\no Following \cite{Bag_Riesz}, we denote $S_{\xi,\eta}$ by $H_{\psi,\phi}^\alpha$. 
As a consequence of Theorem \ref{th_rapp_risol} we get another proof of Proposition 2.1 of \cite{Bag_Riesz}. That is, $H_{\psi,\phi}^{\alpha}$ is a densely defined closed operator and $(H_{\psi,\phi}^{\alpha})^*=H_{\phi,\psi}^{\ol{\alpha}}$. Furthermore, the resolvent set of $H_{\psi,\phi}^{\alpha}$ is the complement of $\ol{\{\alpha_n\}}$,  $ \rho(H_{\psi,\phi}^{\alpha})=\ol{\{\alpha_n\}}^c$.

\no Therefore, if $\inf_{n\in \N} |\alpha_n|>0$ we get the reconstruction formulas for $f\in \H$
\begin{align*}
f=\sum_{n=1}^\infty \alpha_n \pin{f}{{H_{\phi,\psi}^{\ol{\alpha}}}^{-1}\phi_n}\psi_n \;\;\text{ and }\;\; f=\sum_{n=1}^\infty \ol{\alpha_n} \pin{f}{{H_{\psi,\phi}^{\alpha}}^{-1}\psi_n}\phi_n.
\end{align*}


\subsection{Weighted Bessel sequences}

Let $\phi=\{\phi_n\},\psi=\{\psi_n\}$ be Bessel sequences and $\alpha=\{\alpha_n\}$ be a complex sequence. Define $\O^\alpha_{\phi,\psi}$ the sesquilinear form
\begin{equation*}
\O^\alpha_{\phi,\psi} (f,g)=\sum_{n=1}^\infty \alpha_n \pin{f}{\phi_n}\pin{\psi_n}{g}.
\end{equation*}

\no 
Putting $\xi=\{\ol{\alpha_n} \phi_n\}$ and $\eta=\{\psi_n\}$, we can consider $\O^\alpha_{\phi,\psi}$ as $\O_{\xi,\eta}$ on $\dom(\xi)\times \dom(\eta)$. We have $\dom(\eta)=\H$, $\dom(D_\eta)=\H$ and $D_\eta=C_\eta ^*$. 
By Theorem \ref{th_cns_0clos}, the operator associated to $\O_{\xi,\eta}$ on $\dom(\xi) \times \dom(\eta)$ is $D_\eta C_\xi$. As we previously said, $D_\eta C_\xi \subseteq S_{\xi,\eta}$ and Example \ref{exm_DC<S} shows that the converse is not true. \\
Assume that $\dim N(D_\eta)=\dim N(D_\psi)< \infty$. Then $D_\eta C_\xi = S_{\xi,\eta}$. Indeed, for $f\in \dom(S_{\xi,\eta})$ the series $\sum_{n=1}^\infty \alpha_n\pin{f}{\phi_n}\pin{\psi_n}{g} =\pin{S_{\xi,\eta}f}{g}$ is convergent. Since $\dim R(C_\eta)^\perp=\dim N(D_\eta) < \infty$, by Lemma \ref{lem_W^perp}, $\{\pin{f}{\xi_n}\}\in l_2$, i.e., $f\in \dom(C_\xi)$. In conclusion $f\in \dom(D_\eta C_\xi)$.

\section{Sequences as images of an ONB through operators}
\label{sec:Ve_n}


Let $\{e_n\}$ be an ONB and $V$ a densely defined operator of $\H$ such that $e_n\in \dom(V)$ for $n\in \N$. One could consider the sequence $\xi=\{\xi_n\}$ given by
\begin{equation}
\label{Ve_n}
\xi_n=V e_n,\qquad \forall n\in \N. 
\end{equation}
Actually the condition (\ref{Ve_n}) occurs for every sequence $\xi$. Indeed, for a fixed ONB $\{e_n\}$ an operator $V$ can defined on  span$(\{e_n\})$ such that $Ve_n = \xi_n$.
However, in general, the choice of the operator $V$ may not be uniquely determined for a fixed ONB $\{e_n\}$. \\
To establish further properties of a sequence defined by (\ref{Ve_n}) we need to consider the restriction $V_0$ of $V$ to span$(\{e_n\})$.

\begin{pro}
	The following statements hold.
	\begin{enumerate}[label=\emph{(\roman*)},ref={(\roman*)}]
		\item The analysis operator $C_\xi$ has domain $\dom(\xi)=\dom(V_0^*)$ and it is defined as $C_\xi f=\{\pin{V_0^*f}{e_n}\}$ for $f\in \dom(V_0^*)$.
		\item The analysis operator $C_\xi$ is densely defined if and only if $V_0$ is closable.
		\item \label{synth1}
		The synthesis operator $D_\xi$ acts as $D_\xi \{c_n\}=\sum_{n=1}^\infty c_n \xi_n $ on the domain 
		$$
		\dom(D_\xi)=\left \{\{c_n\}\in l_2: V_0 \left(\sum_{n=1}^k c_n e_n \right ) \text{ is convergent in } \H \right \}.
		$$	
	\end{enumerate}		
	Assume that $\dom(\xi)$ is dense, i.e., $V_0$ is closable. The following statements hold.
	\begin{enumerate}[label=\emph{(\roman*)}]
		\item[\emph{(iii')}]  The synthesis operator $D_\xi$ acts as $D_\xi \{c_n\}=\ol{V_0} (\sum_{n=1}^\infty c_n e_n) $ on the domain 
		$$
		\dom(D_\xi)=\left \{\{c_n\}\in l_2: \sum_{n=1}^k c_n e_n \text{ is convergent in } \dom(\ol{V_0})[\nor_{\ol{V_0}}] \right \}.
		$$	
		\item[\emph{(iv)}]  The adjoint $C_\xi^*$ of $C_\xi$ is defined as $C_\xi^* \{c_n\}=\ol{V_0} (\sum_{n=1}^\infty c_n e_n )$ on the domain 
		$$
		\dom(C_\xi^*)=\left \{\{c_n\}\in l_2: \sum_{n=1}^\infty c_n e_n \in \dom(\ol{V_0}) \right \}.
		$$
		\item[\emph{(v)}] The operator $S_\xi$ is defined as $S_\xi f = \ol{V_0}V_0^* f$ on 
		$$
		\dom(S_\xi)=\left \{f\in \H: \sum_{n=1}^k \pin{f}{\xi_n} e_n \text{ is convergent in } \dom(\ol{V_0})[\nor_{\ol{V_0}}] \right \}.
		$$
		\item[\emph{(vi)}] The operator $C_\xi^* C_\xi$ is $\ol{V_0}V_0^*=|V_0^*|^2$.
	\end{enumerate}
\end{pro}
\begin{proof}
	\begin{enumerate}[label={(\roman*)}]
		\item The proof is identical to that of \cite[Proposition II.1]{Bag_sesq}.
		\item It comes from point (i).
		\item[] Points (iii) and (iii') follow by the definition of $D_\xi$.
		\item[(iv)] Let $\{c_n\}\in l_2$. Then 
		\begin{align*}
			\pin{\{c_n\}}{C_\xi f}_2=\sum_{n=1}^\infty c_n \pin{\xi_n}{f}=\pin{\sum_{n=1}^\infty c_n e_n}{V_0^* f}.
		\end{align*}
		Therefore, $\{c_n\}\in \dom(C_\xi^*)$ if and only if $\sum_{n=1}^\infty c_n e_n \in \dom(\ol{V_0})$. Moreover $C_\xi^* \{c_n\}=\ol{V_0} (\sum_{n=1}^\infty c_n e_n )$.
		\item[(v)] It is a consequence of the relation $S_\xi=D_\xi C_\xi$.
		\item[(vi)] Let $f\in \H$. Then $f\in \dom(C_\xi^* C_\xi)$ if and only if $f\in \dom(V_0^*)$ and $V_0^* f=\sum_{n=1}^\infty \pin{f}{\xi_n} e_n \in \dom(\ol{V_0})$, i.e., $f\in \dom(\ol{V_0}V_0^*)$. \qedhere
	\end{enumerate}
\end{proof}

\no In \cite{Classif} some characterizations of sequences are given based on the operator $V_0$. Now assume that $\{e_n\}$ is an ONB, $V,Z$ are operators such that $e_n\in \dom(V)\cap\dom(Z)$ and let $\xi=\{\xi_n\}=\{V e_n\}$ and $\eta=\{\eta_n\}=\{Z e_n\}$. In the next theorem we study the sesquilinear form induced by these sequences. Again $V_0$ and $Z_0$ are the restrictions to span$(\{e_n\})$ of $V$ and $Z$, respectively.

\begin{theo}
	Let $\O_{\xi,\eta}$ be defined on $\dom(\xi)\times \dom(\eta)$. The following statements hold.
	\begin{enumerate}[label=\emph{(\roman*)}]
		\item $\O_{\xi,\eta}(f,g)=\pin{V_0^*f}{Z_0^*g}$ for all $f\in\dom(\xi),g\in \dom(\eta)$.
		\item Assume that $\dom(\eta)$ is dense, i.e., $Z_0$ is closable. The operator associated to $\O_{\xi,\eta}$ is $\ol{Z_0}V_0^*$.
		\item $\O_{\xi,\eta}$ is $0$-closed if and only if $V_0^*, Z_0^*$ are semi-bounded and $R(V_0^*)\dotplus R(Z_0^*)^\perp=\H$ (resp., $R(Z_0^*)\dotplus R(V_0^*)^\perp=\H$).
	\end{enumerate}
\end{theo}
\begin{proof}
	\begin{enumerate}[label={(\roman*)}]
		\item For all $f\in \dom(\xi)=\dom(V_0^*), g\in \dom(\eta)=\dom(Z_0^*)$ 
		$$
		\O_{\xi,\eta}(f,g)=\sum_{n=1}^\infty\pin{f}{\xi_n}\pin{\eta_n}{g}=\sum_{n=1}^\infty\pin{V_0^*f}{e_n}\pin{e_n}{Z_0^*g}=\pin{V_0^*f}{Z_0^*g}.
		$$
		\item It is an immediate consequence of the previous point.
		\item Assume that $\xi,\eta$ are lower semi-frames. Then $R(C_\xi)\dotplus R(C_\eta)^\perp = l_2$ if and only if $R(V_0^*)\dotplus R(Z_0^*)^\perp=\H$. Indeed, $R(V_0^*),R(Z_0^*)$ are closed and the assertion can be obtained by the following considerations 
	\begin{itemize}
		\item $\{d_n\}\in R(C_\xi)\cap R(C_\eta)^\perp$ if and only if $\sum_{n=1}^\infty d_n e_n \in R(V_0^*)\cap R({Z_0}^*)^\perp$;
		\item $\{d_n\}\in R(C_\xi)+ R(C_\eta)^\perp$ if and only if $\sum_{n=1}^\infty d_n e_n \in R(V_0^*)+ R({Z_0}^*)^\perp$.
	\end{itemize}
	The sesquilinear form $\O_{\xi,\eta}$ on $\dom(\xi)\times \dom(\eta)$ is $0$-closed, by Theorem \ref{th_cns_0clos}, if and only if $\xi,\eta$ are lower semi-frames and $R(C_\xi)\dotplus R(C_\eta)^\perp = l_2$, if and only if $V_0^*, Z_0^*$ are semi-bounded and $R(V_0^*)\dotplus R({Z_0}^*)^\perp=\H$.
	The proof is completed noting that, by \cite[Theorem 2.3]{Tang}, $R(V_0^*)\dotplus R({Z_0}^*)^\perp=\H$ is equivalent to $R(V_0^*)\dotplus R({Z_0}^*)^\perp=\H$ when $V_0^*, Z_0^*$ are semi-bounded.  \qedhere
	\end{enumerate}
\end{proof}

\begin{cor}
	Let $\xi=\{\xi_n\}$ be a sequence of $\H$ with $\xi_n=V e_n$, where $\{e_n\}$ is an ONB, $V$ is an operator of $\H$ and $e_n\in \dom(V)$ for all $n\in \N$. Denoting by $V_0$ the restriction of $V$ to span$(\{e_n\})$, then
\begin{enumerate}[label=\emph{(\roman*)}]
	\item $\O_\xi(f,g)=\pin{V_0^*f}{V_0^*g}$ for all $f,g\in\dom(\xi)$.  
	\item $\O_\xi$ is densely defined if and only if $V_0$ is closable. In this case the operator associated to $\xi$ is $\ol{V_0}V_0^*=|V_0^*|^2$.
\end{enumerate}
\end{cor}

\section*{Acknowledgments}

The author thanks prof. Trapani for suggesting the problem and for many useful conversations. This work has been done in the framework of the project “Alcuni aspetti di teoria spettrale di operatori e di algebre; frames in spazi di Hilbert rigged” 2018, of the “National Group for Mathematical Analysis, Probability and
their Applications” (GNAMPA – INdAM).

\vspace*{0.5cm}
\begin{center}
\textsc{Rosario Corso, Dipartimento di Matematica e Informatica} \\
\textsc{Università degli Studi di Palermo, I-90123 Palermo, Italy} \\
{\it E-mail address}: {\bf rosario.corso@studium.unict.it}
\end{center}

\end{document}